\pdfoutput=1
\documentclass[11pt, letterpaper]{article}

\usepackage{fullpage,layout,amsfonts,amsmath,amsthm}

\usepackage{mdframed,thmtools,enumitem} 

\usepackage{amssymb,multirow,makecell,tabularx,tikz}

\definecolor{shadecolor}{gray}{0.95}
\declaretheoremstyle[
headfont=\normalfont\bfseries,
notefont=\mdseries, notebraces={(}{)},
bodyfont=\normalfont,
postheadspace=0.5em,
spaceabove=1pt,
mdframed={
  skipabove=8pt,
  skipbelow=8pt,
  hidealllines=true,
  backgroundcolor={shadecolor},
  innerleftmargin=4pt,
  innerrightmargin=4pt}
]{shaded}

 \usepackage{xcolor,color,graphicx}

\usepackage{algorithm}
\usepackage[noend]{algpseudocode}
\usepackage{verbatim}

\usepackage[algo2e]{algorithm2e}

\newcommand{\bb}{\mathbf{b}}

\newcommand{\bx}{\mathbf{x}}

\newcommand{\bbstar}{\mathbf{b}^\star}

\newcommand{\bs}{\mathbf{s}}

\newcommand{\bv}{\mathbf{v}}

\newcommand{\I}{\mathbf{I}}

\newcommand{\Fix}{\mathsf{Fix}}

\newcommand{\bB}{\mathbf{B}}

\newcommand{\bA}{\mathbf{A}}
\newcommand{\ba}{\mathbf{a}}

\newcommand{\bu}{\mathbf{u}}

\newcommand{\bc}{\mathbf{c}}

\newcommand{\bP}{\mathbf{P}}

\newcommand{\R}{\mathbb{R}}
\newcommand{\N}{\mathbb{N}}
\newcommand{\cI}{\mathcal{I}}
\newcommand{\cR}{\mathcal{R}}
\newcommand{\cS}{\mathcal{S}}

\newcommand{\cO}{\mathcal{O}}
\newcommand{\cT}{\mathcal{T}}
\newcommand{\cA}{\mathcal{A}}

\newcommand{\cF}{\mathcal{F}}
\newcommand{\cG}{\mathcal{G}}

\newcommand{\cJ}{\mathcal{J}}
\newcommand{\cL}{\mathcal{L}}
\newcommand{\cH}{\mathcal{H}}
\newcommand{\cM}{\mathcal{M}}

\newcommand{\cQ}{\mathcal{Q}}

\newcommand{\cK}{\mathcal{K}}

\newcommand{\cTtilde}{\tilde{\mathcal{T}}}

\newcommand{\bbtilde}{\tilde{\mathbf{b}}}

\newcommand{\zer}{\mathsf{zer}}
\newcommand{\dist}{\mathsf{dist}}

\newcommand{\be}{\begin{equation}}
\newcommand{\ee}{\end{equation}}

\usepackage[colorinlistoftodos,bordercolor=orange,backgroundcolor=orange!20,linecolor=orange,textsize=scriptsize]{todonotes}

\DeclareMathOperator{\prox}{prox}       % proximal operator      

\DeclareMathOperator{\Arg}{Arg}         % Argument

\declaretheorem[style=shaded,within=section]{definition}
\declaretheorem[style=shaded,sibling=definition]{theorem}
\declaretheorem[style=shaded,sibling=definition]{proposition}

\declaretheorem[style=shaded,sibling=definition]{corollary}
\declaretheorem[style=shaded,sibling=definition]{lemma}
\theoremstyle{remark}
\newtheorem{example}{Example} %[section]
\newtheorem{remark}{Remark} %[section]
\usepackage[colorlinks=true,linkcolor=blue]{hyperref}

\title{On the  metric resolvent: nonexpansiveness, convergence rates and  applications\thanks{This work was  funded by the National Natural Science Foundation of China under grant no. 62071028. }}
\author{Feng Xue\footnote{National Key Laboratory of Science and Technology on Test Physics and Numerical Mathematics, Beijing, 100076, China (E-mail: \url{fxue@link.cuhk.edu.hk})}}
\date \today

\theoremstyle{definition}

\begin{document}
\maketitle

\begin{abstract} 
In this paper, we study the nonexpansive properties of metric resolvent, and present a convergence rate analysis for the associated fixed-point iterations (Banach-Picard and Krasnosel'skii-Mann types). Equipped with a variable metric, we develop the global ergodic and non-ergodic iteration-complexity bounds in terms of both solution distance and objective value. A byproduct of our expositions also extends the proximity operator and Moreau's decomposition identity to arbitrary variable metric. It is further shown that  many classes of the first-order operator splitting algorithms, including alternating direction methods of multipliers, primal-dual hybrid gradient  and Bregman iterations, can be expressed by the fixed-point iterations of a simple metric resolvent, and thus, the convergence can be analyzed within this unified framework. 
\end{abstract}



\section{Introduction}

\subsection{Resolvent and operator splitting algorithm}
Resolvent plays a fundamental role in the developments of convex analysis and monotone operator theory, since it provides an effective way to replace  finding zeros problem\footnote{$\cA: \cH \mapsto 2^\cH$ is a set-valued maximally monotone operator  \cite[Definition 20.1]{plc_book}, where  $\cH$ denotes a finite-dimensional Hilbert space, especially when the dimensionality is not specified.}:
\be \label{inclusion}
{\bf 0} \in \cA \bbstar
\ee
 with a fixed point problem for resolvent of $\cA$, defined by: $\cJ_\cA = (\cI+ \cA)^{-1}$  \cite[Definition 23.1]{plc_book}, where $\cI$ denotes identity operator \cite{res_1,res_7}.   The resolvent $\cJ_\cA$ is single-valued and firmly nonexpansive, if $\cA$ is maximally monotone \cite[Proposition 23.7, Corollary 23.8, Corollary 23.10]{plc_book}. In particular, 
if $\cA = \partial h$ for some $h \in \Gamma(\cH)$\footnote{$\Gamma(\cH)$ stands for a class of proer, lower semi-continuous and convex functions.}, the resolvent becomes a proximity operator: $\prox_h: \cH \mapsto \cH: \bb \mapsto \arg\min_\bx h(\bx) + \frac{1}{2} \big\|\bx - \bb\big\|^2$ \cite[Eq.(2.13)]{plc}. The classical resolvent and the associated proximity operator have been extensively studied in \cite{plc,plc_book}. In this paper, we extend the classical resolvent to the case of arbitrary variable metric $\cQ$: 
\be \label{t}
\cT:=(\cA+\cQ)^{-1} \cQ = (\cI + \cQ^{-1} \cA)^{-1}
: = \cJ_{\cQ^{-1} \cA}
\ee
and study its properties. 

Nowadays, there has been a revived interest in the design and analysis of the first-order operator splitting algorithms \cite{teboulle_2018}, typically including Douglas--Rachford splitting (DRS) \cite{drs}, alternating direction method of multipliers (ADMM) \cite{glowinski_1,glowinski_2}, primal-dual splitting (PDS) \cite{pdhg,cp_2011} and  Bregman methods \cite{osher_2005,yin_2008,sb,zxq}.  The convergence analysis of these algorithms is often performed case-by-case. Though  some unified frameworks and tools have recently been proposed, e.g. nonexpansive operator \cite{ljw_mapr}, Fej\'{e}r monotonicity \cite{plc_vu}, Bregman proximal framework \cite{teboulle_2018}, assymetric forward-backward-adjoint splitting scheme \cite{latafat}, and others \cite{unified_ieee,beck_unified},  these works do not associate the operator splitting algorithms with a  simple and unified nonexpansive mapping.  The main purpose of this paper is to show that many splitting algorithms can be simply expressed by the metric resolvent \eqref{t} or its generalized/relaxed version.

\subsection{Contributions}
In this paper, we first study the nonexpansiveness, cocoerciveness and areveragedness of the metric resolvent  \eqref{t}. Then, for the associated fixed-point Banach-Picard and  Krasnosel'skii-Mann iterations, we establish the global pointwise/nonergodic and ergodic convergence rates in terms of the solution distance and aysmptotic regularity. The convergence rates in terms of objective function value are further presented, when the  metric resolvent is rewritten as a generalized proximity operator. 

The main results can be applied to many existing operator splitting methods. In particular, we show that a variety of popular algorithms can be uniformly represented by the relaxed metric resolvent, by specifying the monotone operator $\cA$, variable metric $\cQ$, and relaxation matrix $\cM$. This unification and simplification helps to understand various algorithms
with substantially simplified analysis, compared to the original proofs in the literature.

\subsection{Related work}
\vskip.1cm
\paragraph{Relation with \cite{boyd_control,boyd_2014,
hbs_2015}} The DRS or equivalently ADMM for the dual problem have been analyzed in these works, using the tool of  resolvent and reflected resolvent. Our results go much beyond them, by  (1) reinterpreting DRS by a new form of metric resolvent, which is simpler than the classical expression; (2) applying the metric resolvent to many other classes of algorithms.

\vskip.1cm
\paragraph{Relation with \cite{pfbs_siam,ljw_mapr,ywt_2017}} The more complicated operator splitting algorithms may be reformulated as a simple fixed-point iteration of  nonexpansive mapping \cite{ljw_mapr}, e.g. three-operator splitting algorithm \cite{ywt_2017}, GFBS \cite{pfbs_siam}. In these works, the nonexpansive mapping $\cT$ depends on specific algorithm, for which the nonexpansive properties have to be analyzed case-by-case. Our work differs from them in that: we always use the same form of the generalized metric resolvent, which provides a unified treatment of many classes of algorithms, for which the convergence behaviours can be immediately obtained by the nonexpansiveness of the metric resolvent.

\vskip.1cm
\paragraph{Relation with \cite{hbs_siam_2012,hbs_prs,hbs_siam_2012_2,hbs_yxm_2015,
hbs_jmiv_2017}} In these works, the authors revisited DRS and PDHG within the framework of proximal point algorithm (PPA) by reformulation of variational inequality. Based on the metric resolvent, our exposition takes a much simpler, yet somewhat equivalent route to tackle the PPA. Since that the fixed-point iteration of metric resolvent is essentially a PPA, we explore more intricate properties related to the strong monotonicity and objective value, which are never discussed in the literature.

\vskip.1cm
\paragraph{Relation with \cite{osher_2005,cjf_1,cjf_2,zxq,sb}} The Bregman iteration and many variants have been developed in the literature. In this work, we revisit these algorithms, and show that they also fall into the category of the metric resolvent, and thus, the convergence analysis is substantially simplified, compared to the original proofs in the literature.

\vskip.1cm
\paragraph{Relation with \cite{pdhg,esser,cp_2011}} We also revisit the PDHG algorithms developed in these works, and show that the PDHG is essentially equivalent to a metric resolvent. More convergence properties are investigated. 

\vskip.1cm
\paragraph{Relation with \cite{lotito,qian,bonnans}} These early works extended the classical PPA \cite{martinet,rtr_1976} to the variable metric case.  In this paper, we tackle this method from the perspective of metric resolvent, and show the wide range of applications to the operator splitting algorithms.

\subsection{Notations} \label{sec_notation}
We use standard notations and concepts from convex analysis and variational analysis, which, unless otherwise specified, can all be found in the classical and recent monographs \cite{rtr_book,rtr_book_2,plc_book,beck_book}.

\vskip.1cm
A few more words about our  notations are in order. The classes of positive semi-definite (PSD) and positive definite (PD) matrices are denoted by $\cM^+$ and $\cM^{++}$, respectively.  The classes of symmetric, symmetric and PSD, symmetric and PD matrices are denoted by $\cM_\cS$, $\cM_\cS^+$, and $\cM_\cS^{++}$, respectively.  For our specific use,  the $\cQ$-based inner product (where $\cQ$ is an arbitrary square matrix) is defined as: $\langle \ba | \bb \rangle_\cQ :=  \langle \cQ \ba | \bb \rangle =  \langle  \ba | \cQ^\top \bb \rangle$, $\forall (\ba,\bb) \in \cH \times \cH$; the $\cQ$-norm is defined as: $\|\ba\|_\cQ^2 :=  \langle \cQ\ba | \ba \rangle$, $\forall \ba \in \cH$. Note that unlike the conventional treatment in the literature,  $\cQ$ is {\it not} assumed to be symmetric and PSD here, and hence, $\|\cdot\|_\cQ$ is not always well--defined.

\vskip.1cm
Note that our expositions  in this paper are largely based on the basic properties of nonexpansive operators in the context of   arbitrary variable metric $\cQ$, which have been thoroughly discussed in  \cite{fxue_1}. For sake of completeness and convenience, a key notion of {\it  $\cQ$--based $\xi$--Lipschitz $\alpha$--averaged} is restated here.
\begin{definition} \label{def_lip} 
{\rm  \cite[Definition 2.2]{fxue_1}
An operator $\cT: \cH \mapsto \cH$ is said to be {\it  $\cQ$--based $\xi$--Lipschitz $\alpha$--averaged}   with $\xi \in \ ]0, +\infty[$ and $\alpha\in \ ]0, 1[$, denoted by $\cT \in \cF^\cQ_{\xi,\alpha}$,  if there exists a $\cQ$--based $\xi$--Lipschitz continuous operator $\cK: \cH \mapsto \cH$, such that $\cT = (1-\alpha) \cI + \alpha \cK$. In particular, if $\xi \in \ ]1, +\infty[$, $\cT$ is {\it $\cQ$--weakly averaged}; if $\xi \in\  ]0,1]$,  $\cT$ is {\it $\cQ$--strongly averaged}.
}
\end{definition}

\section{The metric resolvent}
\label{sec_resolvent}
This section investigates the nonexpansiveness of the metric resolvent and the convergence properties of its associated Banach-Picard iteration. Many results in this part follow from  \cite[Sections 2, 3]{fxue_1}.

\subsection{Nonexpansive properties}
The basic properties of metric resolvent \eqref{t} are summarized in the following lemma.

\begin{lemma} [Nonexpansiveness of metric resolvent] 
\label{l_t}
Given $\cT$ defined in \eqref{t}, then  the following hold.

{\rm (i)}  $\cT$  is $\cQ$--partly nonexpansive (i.e. $\cQ$--based 1--cocoercive). 

{\rm (ii)}  If $\cQ\in \cM_\cS$,  $\cT$ is  $\cQ$--firmly nonexpansive (i.e. $\cT \in \cF^\cQ_{1, \frac{1}{2}}$).

{\rm (iii)} If $\cQ\in \cM_\cS^{++}$, $\cT$ can  be equivalently expressed as:
 \[
\cT =  \cJ_{\cQ^{-1} \circ \cA} = \cQ^{-1} \circ 
\cJ_{ \cA \circ \cQ^{-1} } \circ \cQ =
\cQ^{-\frac{1}{2}} \circ \cJ_{\cQ^{-\frac{1}{2}} \circ \cA \circ \cQ^{-\frac{1}{2}} } \circ \cQ^{\frac{1}{2}}   
 \]
\end{lemma}
\begin{proof} 
(i) By monotonicity of $\cA$, we have:
\begin{eqnarray}
&&  \big\| \cT \bb_1 - \cT \bb_2  \big\|_\cQ^2 
\nonumber \\  &=&  \langle \cQ \cT \bb_1 - \cQ \cT \bb_2
\big|  \cT\bb_1-\cT\bb_2 \rangle 
\nonumber \\ 
&  \le &\langle \cQ \cT \bb_1 - \cQ \cT \bb_2 | \cT\bb_1-\cT\bb_2 \rangle + 
\langle \cA \cT \bb_1 - \cA \cT \bb_2 | \cT\bb_1-\cT\bb_2 \rangle  \quad \text{by monotone $\cA$}
\nonumber \\ &= & 
 \langle \bb_1 - \bb_2 | \cT\bb_1-\cT\bb_2 \rangle_\cQ
 \quad \text{since $\cQ\cT +\cA \cT = \cQ$ by \eqref{t}} 
 \nonumber 
\end{eqnarray}

\vskip.1cm
(ii) \cite[Lemma 2.3]{fxue_1}.

\vskip.1cm
(iii) The first equality: clear.

The second equality: $\ba = \cT \bb = (\cA +\cQ)^{-1} \cQ \bb \Longrightarrow 
\cQ \bb \in \cA \cQ^{-1} \cQ \ba + \cQ \ba = 
(\cI + \cA \circ \cQ^{-1})  \cQ \ba  \Longrightarrow 
\cQ \ba =(\cI + \cA \circ \cQ^{-1})^{-1} (\cQ \bb) 
\Longrightarrow  \ba  = \cQ^{-1} 
(\cI + \cA \circ \cQ^{-1})^{-1} (\cQ \bb)  $.

The third equality: 
$\ba = \cT \bb = (\cA +\cQ)^{-1} \cQ \bb \Longrightarrow 
\cQ \bb \in \cA  \ba + \cQ \ba  \Longrightarrow 
\cQ^{\frac{1}{2}} \bb \in \cQ^{-\frac{1}{2}} \cA \ba + 
\cQ^{\frac{1}{2}} \ba = \cQ^{-\frac{1}{2}} \cA
\cQ^{-\frac{1}{2}} \cQ^{\frac{1}{2}} \ba + 
\cQ^{\frac{1}{2}} \ba = (\cQ^{-\frac{1}{2}} \cA
\cQ^{-\frac{1}{2}} +\cI)  \cQ^{\frac{1}{2}} \ba 
\Longrightarrow \cQ^{\frac{1}{2}} \ba = (\cQ^{-\frac{1}{2}} \cA \cQ^{-\frac{1}{2}} +\cI) ^{-1}  \cQ^{\frac{1}{2}} \bb 
\Longrightarrow \ba = \cQ^{-\frac{1}{2}}  
(\cQ^{-\frac{1}{2}} \cA \cQ^{-\frac{1}{2}} +\cI) ^{-1}  
(\cQ^{\frac{1}{2}}    \bb) $.
\end{proof}

\vskip.2cm
Similar to Lemma \ref{l_t}, the complementary operator $\cR:=\cI - \cT$ has the following properties.

\begin{lemma} [Nonexpansiveness of $\cR$]
 \label{l_r}
The operator $\cR = \cI - \cT$ satisfies: 

{\rm (i)}  $\cR$ is $\cQ^\top$--partly nonexpansive  (i.e. $\cQ^\top$--based 1--cocoercive).

{\rm (ii)} If $\cQ\in \cM_\cS$,  $\cR$ is $\cQ$--firmly nonexpansive (i.e. $\cR \in \cF^\cQ_{1, \frac{1}{2}}$).

{\rm (iii)} If $\cQ\in \cM_\cS^{++}$, $\cR$ can be equivalently expressed as:
\[
\cR =  \cJ_{\cA^{-1} \circ \cQ} = 
\cQ^{-1} \circ  \cJ_{\cQ \circ \cA^{-1}} \circ \cQ  =
 \cQ^{-\frac{1}{2}} \circ \cJ_{  \cQ^{\frac{1}{2}} \circ \cA^{-1} \circ \cQ^{\frac{1}{2}} } \circ \cQ^{\frac{1}{2}}
\]
\end{lemma}

\begin{proof} 
(i)--(ii)  \cite[Lemma 2.3]{fxue_1}.

\vskip.1cm
(iii)  The first equality:  $\ba=\cT\bb = (\cA +\cQ)^{-1} \cQ \bb \Longrightarrow \cQ(\bb-\ba) \in \cA \ba  \Longrightarrow \ba \in \cA^{-1} \cQ (\bb-\ba) \Longrightarrow \bb -\ba \in \bb - \cA^{-1}  \cQ (\bb-\ba)  \Longrightarrow \cR \bb = (\cI-\cT)\bb = \bb -\ba  = (\cI + \cA^{-1} \circ \cQ)^{-1}\bb$. 

The second equality: $\ba = \cR \bb =(\cI + \cA^{-1} \circ \cQ)^{-1}\bb \Longrightarrow 
 \bb \in  \ba + \cA^{-1} \cQ \ba   \Longrightarrow 
 \cQ \bb \in  \cQ \ba + \cQ \cA^{-1} \cQ \ba 
 \Longrightarrow \cQ \ba = (\cI+ \cQ \cA^{-1})^{-1} \cQ \bb
 \Longrightarrow  \ba = \cQ^{-1} (\cI+ \cQ \cA^{-1})^{-1} \cQ \bb $.

The third equality: $\ba = \cR \bb =(\cI + \cA^{-1} \circ \cQ)^{-1}\bb \Longrightarrow 
 \bb \in  \ba + \cA^{-1} \cQ \ba   \Longrightarrow 
 \cQ^{\frac{1}{2}} \bb \in \cQ^{\frac{1}{2}} \ba + \cQ^{\frac{1}{2}} \cA^{-1} \cQ^{\frac{1}{2}} \cQ^{\frac{1}{2}} \ba 
 \Longrightarrow \cQ^{\frac{1}{2}} \ba = (\cI+ \cQ^{\frac{1}{2}} \cA^{-1} \cQ^{\frac{1}{2}} )^{-1} \cQ^{\frac{1}{2}} \bb
 \Longrightarrow  \ba = \cQ^{-\frac{1}{2}} (\cI+ \cQ^{\frac{1}{2}} \cA^{-1} \cQ^{\frac{1}{2}} )^{-1} \cQ^{\frac{1}{2}} \bb$.
\end{proof}

If $\cA$ is $\mu$-strongly monotone,  Lemmas \ref{l_t} and \ref{l_r} can be strengthened as follows.
\begin{lemma}  [Cocoerciveness and averagedness]
\label{l_strong}
Given $\cT$ defined in \eqref{t} and $\cR = \cI - \cT$, if $\cA$ is $\mu$-strongly  monotone, $\cQ \in \cM^{++}$, then, the following hold.

\vskip.1cm
{\rm (i)}  $\cT$ is $\cQ$--based $(1+\frac{\mu}{\|\cQ\|})$--cocoercive, $\cR$ is $\cQ^\top$--based 1--cocoercive.

{\rm (ii)}  $\cT \in \cF^\cQ_{\frac{\|\cQ\|}{2\mu + \|\cQ\|}, \frac{2\mu +\|\cQ\|} {2\mu +2\|\cQ\|} }$,    $\cR \in \cF^\cQ_{1, \frac{\|\cQ\|} { 2(\|\cQ\| +  \mu ) } }$, 
if $\cQ\in \cM_\cS^+$. 

{\rm (iii)} Both $\cT$ and  $\cR$ are $\cQ$--firmly nonexpansive, if $\cQ \in \cM_\cS^+$.

{\rm (iv)} If $\cQ \in \cM_\cS^{++}$,  $\cR$ satisfies:
\[
\big\langle \bb_1 -  \bb_2\big| 
  \cR\bb_1 -  \cR \bb_2 \big\rangle_\cQ 
  \ge \frac{\|\cQ\| + \mu} { \|\cQ\| +2\mu } 
  \big\| \cR\bb_1 - \cR \bb_2 \big\|_\cQ^2
+ \frac{ \mu} { \|\cQ\| +2\mu } 
  \big\| \bb_1 -  \bb_2 \big\|_\cQ^2
\]

\end{lemma}

\begin{proof}
(i) By the strong monotonicity of $\cA$, i.e.  $\langle \cA \bb_1 - \cA \bb_2 | \bb_1 - \bb_2 \rangle \ge \mu \|\bb_1 -\bb_2\|^2$, $\forall (\bb_1,\bb_2) \in \cH \times \cH$, we have, from Lemma \ref{l_t}--(i), that:
\[
 \big\| \cT \bb_1 - \cT \bb_2  \big\|_\cQ^2
 + \mu \big\| \cT \bb_1 - \cT \bb_2 \big\|^2 \le 
 \langle \bb_1 - \bb_2 | \cT\bb_1-\cT\bb_2 \rangle_\cQ
 \]
Noting that $ \big\| \cT \bb_1 - \cT \bb_2 \big\|^2
\ge \frac{1}{\|\cQ\|}  \big\| \cT \bb_1 - \cT \bb_2 \big\|_\cQ^2$ by $\cQ \in \cM^+$, it yields:
\[
\big \langle \bb_1 - \bb_2 \big| \cT\bb_1-\cT\bb_2 
\big \rangle_\cQ
 \ge \Big( 1+\frac{\mu}{\|\cQ\|} \Big) 
 \big\| \cT \bb_1 - \cT \bb_2  \big\|_\cQ^2 
\]
The cocoerciveness of $\cR$ is obtained by \cite[Lemma 2.11--(ii)]{fxue_1}, since $1+\frac{\mu}{\|\cQ\|} > 1$. 

\vskip.2cm
(ii) \cite[Lemma 2.11--(i)]{fxue_1} or letting $\gamma=1$ in \cite[Theorem 2.8--(iv)]{fxue_1}.

\vskip.2cm
(iii) \cite[Theorem 2.8--(ii)]{fxue_1}.

\vskip.2cm
(iv) follows from (i), by several simple algebraic manipulations.
\end{proof}

\subsection{Generalized proximity operator}
If $\cA$ is {\it cyclically} maximally monotone, then, $\exists h \in \Gamma(\cH)$, such that $\cA = \partial h$ \cite[Theorem 22.14]{plc_book}. Thus, by Fermat's rule \cite[Theorem 16.2]{plc_book},  the monotone inclusion ${\bf 0} \in \cA \bbstar$ is equivalent to finding a minimizer of $h$, i.e. $\bbstar \in \zer \partial h = \Arg\min h$. The metric resolvent \eqref{t} becomes a generalized proximity operator: $\prox_h^\cQ: \cH \mapsto \cH: \bb \mapsto \arg\min_\bx h(\bx) + \frac{1}{2} \big\|\bx - \bb\big\|^2_\cQ$  \cite[Definition 2.3]{pesquet_2016}. Then, the generalized  proximity operator is connected to the ordinary one $\prox_h$ via  the following result. 
\begin{proposition} \label{p_id_prox}
Given $\cT$ defined by \eqref{t} and $\cR : =\cI - \cT$, if $\cA = \partial h$, then, the following hold.

{\rm (i)} $\cT =   \cQ^{-\frac{1}{2}} \circ \prox_{ h \circ \cQ^{-\frac{1}{2}} } \circ \cQ^{\frac{1}{2}} = \prox_h^\cQ$.

{\rm (ii)} $\cR =   \cQ^{-\frac{1}{2}} \circ \prox_{ h^* \circ \cQ^{\frac{1}{2}} } \circ \cQ^{\frac{1}{2}} = \cQ^{-1} \circ \prox_{h^*}^{\cQ^{-1}} \circ \cQ$.

\end{proposition}

\begin{proof}
(i) By Lemma \ref{l_t}--(iii) and \cite[Proposition 16.34, Example 23.3]{plc_book}, we have:
\[
\cT =   \cQ^{-\frac{1}{2}} \circ \cJ_{\cQ^{-\frac{1}{2}} \circ 
\partial h \circ \cQ^{-\frac{1}{2}} } \circ \cQ^{\frac{1}{2}}   
=  \cQ^{-\frac{1}{2}} \circ \cJ_{ \partial (h \circ \cQ^{-\frac{1}{2}} ) } \circ \cQ^{\frac{1}{2}}   
= \cQ^{-\frac{1}{2}} \circ \prox_{ h \circ \cQ^{-\frac{1}{2}} } \circ \cQ^{\frac{1}{2}}
\]
which implies that
\begin{eqnarray}
\cT: \bb & \mapsto & \cQ^{-\frac{1}{2}} \Big( \arg\min_\bu 
h(\cQ^{-\frac{1}{2}} \bu) + \frac{1}{2} \big\|\bu - \cQ^{\frac{1}{2}} \bb \big\|^2  \Big)
\nonumber \\
&= & \arg\min_\bv h( \bv) + \frac{1}{2} \big\| \cQ^{\frac{1}{2}}\bv - \cQ^{\frac{1}{2}} \bb \big\|^2 
\quad \text{by changing variable: $\bv = \cQ^{-\frac{1}{2}} \bu$}
\nonumber \\
&  = & \arg\min_\bv h( \bv) + \frac{1}{2} \big\| \bv - \bb \big\|_\cQ^2 
\nonumber 
\end{eqnarray} 

(ii) By Lemma \ref{l_r}--(iii) and \cite[Proposition 16.9]{plc_book}, we have:
\[
\cR =   \cQ^{-\frac{1}{2}} \circ \cJ_{  \cQ^{\frac{1}{2}} \circ \partial h^* \circ \cQ^{\frac{1}{2}} } \circ \cQ^{\frac{1}{2}}  
=   \cQ^{-\frac{1}{2}} \circ \cJ_{ \partial (h^* \circ \cQ^{\frac{1}{2}} ) } \circ \cQ^{\frac{1}{2}} 
= \cQ^{-\frac{1}{2}} \circ \prox_{ h^* \circ \cQ^{\frac{1}{2}} } \circ \cQ^{\frac{1}{2}}
\]
which implies that
\begin{eqnarray}
\cR: \bb & \mapsto & \cQ^{-\frac{1}{2}} \Big( \arg\min_\bu 
h^*(\cQ^{\frac{1}{2}} \bu) + \frac{1}{2} \big\|\bu - \cQ^{\frac{1}{2}} \bb \big\|^2  \Big)
\nonumber \\
&= & \cQ^{-1} \Big( \arg\min_\bv h^*( \bv) + \frac{1}{2} \big\| \cQ^{-\frac{1}{2}}\bv - \cQ^{\frac{1}{2}} \bb \big\|^2 \Big)
\quad \text{by changing variable: $\bv = \cQ^{\frac{1}{2}} \bu$}
\nonumber \\
&  = &  \cQ^{-1} \Big( \arg\min_\bv h^*( \bv) + \frac{1}{2} \big\| \bv - \cQ \bb \big\|_{\cQ^{-1}}^2 \Big)
\nonumber \\
&  = & \big( \cQ^{-1} \circ \prox_{h^*}^{\cQ^{-1}} \circ \cQ
\big)   \bb
\nonumber 
\end{eqnarray} 
\phantom{s}
\end{proof}

An important corollary immediately follows from Proposition \ref{p_id_prox}, which  generalizes the classical Moreau's identity \cite[Proposition 23.18, Theorem 14.3-(ii)]{plc_book} from $\cQ = \tau^{-1} \cI$ with $\tau >0$ to arbitrary metric $\cQ$.  
\begin{corollary} [Generalized Moreau's identity]
Given a monotone operator $\cA$, a function $h \in \Gamma(\cH)$ and a metric $\cQ \in \cM_\cS^{++}$, the identity operator $\cI$ can be equivalently decomposed as:
\[
{\rm (i)} \ \cI = \cJ_{\cQ^{-1} \cA} + \cJ_{\cA^{-1} \cQ} 
= \cJ_{\cQ^{-1} \cA} +
\cQ^{-1} \circ  \cJ_{\cQ \circ \cA^{-1}} \circ \cQ 
= \cQ^{-1} \circ \cJ_{ \cA \circ \cQ^{-1} } \circ \cQ 
+ \cJ_{\cA^{-1} \cQ} 
\]
\[
{\rm (ii)} \ \cI = \cQ^{-\frac{1}{2}} \circ \prox_{ h \circ \cQ^{-\frac{1}{2}} } \circ \cQ^{\frac{1}{2}} +\cQ^{-\frac{1}{2}} \circ \prox_{ h^* \circ \cQ^{\frac{1}{2}} } \circ \cQ^{\frac{1}{2}} 
=  \prox_h^\cQ+  \cQ^{-1} \circ \prox_{h^*}^{\cQ^{-1}} \circ \cQ
\]
\end{corollary}
\begin{proof}
(i) Lemma \ref{l_t}--(iii) and Lemma \ref{l_r}--(iii);

\vskip.1cm
(ii) Proposition \ref{p_id_prox}.
\end{proof}

\section{The Banach-Picard iteration of metric resolvent}
\subsection{Scheme}
The scheme is given as:
\be  \label{ppa}
\bb^{k+1} :=\cT \bb^k,  \quad 
\text{where\ \ }\cT:=(\cA+\cQ)^{-1} \cQ
\ee
which is equivalent to the monotone inclusion ${\bf 0} \in \cA \bb^{k+1} + \cQ (\bb^{k+1} - \bb^k)$. It  takes a typical  variable metric PPA form \cite{bredies_2017,lotito,qian,bonnans}.

If $\cA = \partial h$, \eqref{ppa} is equivalent to finding a minimizer of $h$: $\bb^{k+1} : = \prox_h^\cQ (\bb^k)$. In particular, if  $\cQ= \frac{1}{ \tau} \cI$, \eqref{ppa} reduces to the classical PPA: $\bb^{k+1} :=\prox_{\tau h} (\bb^k)$, whose convergence properties have been well studied in  \cite{passty,rtr_1976,plc,boyd_prox}.

\subsection{Convergence analysis}
Regarding the convergence of  \eqref{ppa},  the next corollary is a straightforward result of \cite[Theorem 3.3-(i), (ii), (iii)]{fxue_1}, which extends \cite[Proposition 6.5.1]{bert_book} and \cite[Theorem 27.1]{plc_book}  to the scheme \eqref{ppa} with  arbitrary  metric $\cQ\in \cM_\cS^{++}$.

\begin{proposition}[Convergence in terms of $\cQ$--based distance] \label{c_ppa}
Let $\bb^0\in \cH$, $\{\bb^k\}_{k \in \N}$ be a sequence generated by \eqref{ppa}. If $\cQ\in \cM_\cS^{++}$, then, the following hold.

{\rm (i)} $\cT$ is $\cQ$--asymptotically regular.

{\rm (ii) [Basic convergence]}  There exists $\bbstar \in \zer \cA$, such that $\bb^k \rightarrow \bbstar$, as $k\rightarrow \infty$.

{\rm (iii) [Sequential error]}   $\|\bb^{k+1 } -\bb^{k} \|_\cQ$ has the pointwise sublinear convergence rate of $\cO(1/\sqrt{k})$:
\[
\big\|\bb^{k +1} -\bb^{k} \big\|_\cQ
\le \frac{1}{\sqrt{k+1}} 
\big\|\bb^{0} -\bb^\star \big\|_\cQ,
\forall k \in \N
\]
\end{proposition}

\begin{proof} 
First, we claim that $\Fix\cT = \zer \cA$. Indeed,  $\bbstar \in \Fix\cT \Longleftrightarrow \bbstar = (\cA +\cQ)^{-1} \cQ \bbstar 
\Longleftrightarrow  \cQ\bbstar \in (\cA +\cQ) \bbstar
\Longleftrightarrow \bf 0 \in \cA \bbstar 
\Longleftrightarrow  \bbstar \in \zer \cA $.

Note that $\cT \in \cF^\cQ_{1, \frac{1}{2}}$ by Lemma \ref{l_t}--(ii). Substituting $\xi=1$ and $\alpha=1/2$ in \cite[Theorem 3.3-(iii)]{fxue_1} completes the proof. 
\end{proof}

We further deduce the  convergence results  of objective value.
\begin{theorem}[Convergence in terms of objective value] \label{t_ppa}
Under the conditions of Proposition \ref{c_ppa}, if $\exists h \in \Gamma(\cH)$, such that $\cA = \partial h$,  then, the following hold.

{\rm (i) [Basic convergence]} The sequence $\{h(\bb^k)\}_{k \in \N}$ is non-increasing, and converges to its minimum, which is attained at some point $\bbstar \in \Arg\min h$.  

{\rm (ii) [Ergodic rate]} The objective value $h(\bb^k)$ converges to $h(\bbstar)$ with an {\it ergodic}  rate of $\cO(1/k)$, i.e.  
\[
   h\Big(\frac{1}{k} \sum_{i=1}^k \bb^i \Big) -  h(\bbstar)  \le \frac{1}{2k}  
    \big\|  \bb^0 - \bbstar \big\|_\cQ^2 
\]

{\rm (iii) [Non-ergodic rate]} The objective value $h(\bb^k)$ converges to $h(\bbstar)$ with the {\it non-ergodic}  rate of $\cO(1/k)$, i.e. 
\[
   h(\bb^{k}) -  h(\bbstar)  \le \frac{1}{2k}  
    \big\|  \bb^0 - \bbstar \big\|_\cQ^2 
\]

{\rm (iv) [Precise estimate of sequential decreasing]} 
$ h(\bb^k) - h(\bb^{k+1})  $ satisfies:
\[
\big\|\bb^{k+1} - \bb^{k} \big\|_\cQ^2
\le h(\bb^k) - h(\bb^{k+1})  \le 
 \frac{3}{2} \big\|\bb^{k-1} - \bb^{k} \big\|_\cQ^2 
- \frac{1}{2} \big\|\bb^{k+1} - \bb^{k} \big\|_\cQ^2
\]
\end{theorem}

\begin{proof}
(i) First, by \cite[Proposition 26.1]{plc_book} and Fermat's rule \cite[Theorem 16.2]{plc_book}, we have, $\bbstar \in \zer\cA = \zer \partial h \Longrightarrow 
\bbstar \in \Arg \min h$, for $h \in \Gamma(\cH)$.

Further, we obtain,  $\forall \bb\in \cH$:
\begin{eqnarray} \label{temp_1}
h(\bb) & \ge & h(\bb^{k+1}) + \langle \partial h(\bb^{k+1}),
\bb - \bb^{k+1} \rangle \quad 
\text{by convexity of $h$}
\nonumber \\
&=& h(\bb^{k+1}) + \langle  \bb^k - \bb^{k+1},
\bb - \bb^{k+1} \rangle_\cQ   \quad 
\text{by \eqref{ppa}}
\end{eqnarray}
Taking $\bb=\bb^k$ in \eqref{temp_1} yields:
\be \label{temp_2}
h(\bb^k)  \ge h(\bb^{k+1}) + \big\|  \bb^k - \bb^{k+1} 
\big\|_\cQ^2 \ge h(\bb^{k+1}),\quad 
\forall k \in \N
\ee
since $\cQ \in \cM_\cS^{++}$. By the similar argument of \cite[Proposition 6.5.1]{bert_book}, it follows that the sequence $\{h(\bb^k)\}_{k\in\N}$ is non-increasing, and converges to $\min_\bb h(\bb)  = h(\bbstar)$.

\vskip.2cm
(ii) Taking $\bb=\bb^\star \in \Arg \min h$ in \eqref{temp_1}, we have:
\begin{eqnarray}
h(\bb^\star) & \ge &  h(\bb^{k+1})  + \langle  \bb^k - \bb^{k+1}, \bb^\star - \bb^{k+1} \rangle_\cQ
\nonumber \\
&=& h(\bb^{k+1}) +  \frac{1}{2} \big\| \bb^{k+1} - \bb^\star \big\|_\cQ^2
-  \frac{1}{2} \big\| \bb^{k } - \bb^\star \big\|_\cQ^2
+ \frac{1}{2} \big\| \bb^{k+1} - \bb^k \big\|_\cQ^2
\nonumber \\
&\ge & h(\bb^{k+1}) +  \frac{1}{2} \big\| \bb^{k+1} - \bb^\star \big\|_\cQ^2
-  \frac{1}{2} \big\| \bb^{k } - \bb^\star \big\|_\cQ^2\ 
\text{since $\cQ\in \cM_\cS^+$}
\nonumber  
\end{eqnarray}
Summing up from $k=0$ to $K-1$, we have:
\be \label{temp_3}
 \sum_{k=0}^{K-1} \big(   h(\bb^{k+1}) -  h(\bbstar) \big) 
\le \frac{1}{2} \big \|  \bb^0 - \bbstar \big\|_\cQ^2 
- \frac{1}{2}  \big\| \bb^{K } - \bbstar \big\|_\cQ^2 
\le \frac{1}{2} \big \|  \bb^0 - \bbstar \big\|_\cQ^2 
\ee
Dividing by $K$ on both sides obtains:
$ \frac{1}{K} \sum_{k=1}^{K}   h(\bb^{k}) -  h(\bbstar)  
\le \frac{1}{2K} \big \|  \bb^0 - \bbstar \big\|_\cQ^2 $,
then, the ergodic  rate is obtained by  $ h\big(\frac{1}{K} \sum_{k=1}^{K} \bb^k \big) \le \frac{1}{K} \sum_{k=1}^{K}   h(\bb^{k})$, due to the convexity of $h$.

\vskip.2cm
(iii) Also notice that $0 \le   h(\bb^{k+1}) -  h(\bbstar) \le 
  h(\bb^{k}) -  h(\bbstar) $ from \eqref{temp_2}, we have 
  $ \sum_{k=0}^{K-1} \big(   h(\bb^{k+1}) -  h(\bbstar) \big)
  \ge K  \big(   h(\bb^{K}) -  h(\bbstar) \big)  $. Then, the non-ergodic  rate  follows from \eqref{temp_3}.

\vskip.2cm
(iv) To obtain a more precise estimate of $h(\bb^k) - h(\bb^{k+1})$, we further derive:
\begin{eqnarray}
h(\bb^k) - h(\bb^{k+1}) 
& \le   &  \langle \partial h(\bb^{k}), 
\bb^{k} - \bb^{k+1} \rangle    
  = \langle \bb^{k-1} - \bb^k, \bb^k - \bb^{k+1} \rangle_\cQ 
\nonumber \\
& =  & \frac{1}{2} \big\|\bb^{k-1} - \bb^{k+1} \big\|_\cQ^2 
- \frac{1}{2} \big\|\bb^{k-1} - \bb^{k} \big\|_\cQ^2
- \frac{1}{2} \big\|\bb^{k} - \bb^{k+1} \big\|_\cQ^2
\nonumber
\end{eqnarray}
Combining with the fact that {\it the sequence  $\|\bb^k-\bb^{k+1}\|_\cQ$ is  non-increasing}, we have:
\[
\big\|\bb^{k-1} - \bb^{k+1} \big\|_\cQ \le 
\big\|\bb^{k-1} - \bb^{k} \big\|_\cQ + 
\big\|\bb^{k} - \bb^{k+1} \big\|_\cQ \le 
2 \big\|\bb^{k-1} - \bb^{k} \big\|_\cQ
\]
which yields:
\[
h(\bb^k) - h(\bb^{k+1})  \le 
 \frac{3}{2} \big\|\bb^{k-1} - \bb^{k} \big\|_\cQ^2 
- \frac{1}{2} \big\|\bb^{k} - \bb^{k+1} \big\|_\cQ^2
\]
Combining with \eqref{temp_2} gives (iv).
\end{proof}

\subsection{The case of $\mu$-strongly monotone $\cA$}
\label{sec_ppa_strong}
The {\it linear convergence} can be reached due to the strongly monotone $\cA$, as stated below. 
\begin{proposition}[Linear convergence of $\cQ$--based distance] \label{p_ppa_strong}
Under the conditions of Proposition \ref{c_ppa}, if $\cA$ is $\mu$-strongly monotone, then,  the following hold.

{\rm (i) [$q$--linear convergence]} Both $\|\bb^{k} -\bbstar \|_\cQ$ and $\|\bb^{k} -\bb^{k+1} \|_\cQ$  are $q$--linearly convergent with the rate of $\sqrt{\frac{\|\cQ\|} {\|\cQ\| +  2\mu}} $.

{\rm (ii) [$r$--linear convergence]}  If $\mu \ge 
\frac{ \sqrt{5} -1 } {4} \|\cQ\| $, 
 $\big\| \bb^k - \bb^{k+1} \big\|_\cQ$  is globally $r$--linearly convergent w.r.t. $\big\| \bb^0 - \bbstar \big\|_\cQ$:
\[
    \big\| \bb^k - \bb^{k+1} \big\|_\cQ
   \le\sqrt{  \frac{2 \mu } { 2 \mu +\|\cQ\| } }
    \cdot \Big( 1+ \frac{2\mu} {\|\cQ\|} \Big) 
    ^{-\frac{k+1}{4} }  
\big\| \bb^0 - \bbstar \big\|_\cQ
\]
If $\mu \in \ \big] 0, \frac{ \sqrt{5} -1 } {4} \|\cQ\|\big[ $, the above result is also locally satisfied, for  $k \ge  \frac{\ln((1+\sqrt{5})/2)}  {\ln \sqrt{1+\frac{2\mu}{\|\cQ\|}} } - 1$. 
\end{proposition}

\begin{proof} 
By Lemma \ref{l_strong}--(ii),  $\cT \in \cF^\cQ_{\frac{\|\cQ\|}{2\mu + \|\cQ\|}, \frac{2\mu +\|\cQ\|} {2\mu +2\|\cQ\|} }$. The proof is completed by substituting $\xi = \frac{\|\cQ\|} {2\mu + \|\cQ\|}$,
 $\alpha =  \frac{2\mu +\|\cQ\|} {2\mu +2\|\cQ\|}$ into \cite[Theorem 3.3--(iv), (v)]{fxue_1},  noting that  $\xi = \frac{1-\alpha}{\alpha}$ and $\nu = 1-\alpha +\alpha \xi^2 = \xi$.
\end{proof}

\begin{remark}
Proposition \ref{p_ppa_strong} can also be proved by  \cite[Proposition 3.4-(v)]{fxue_1} and Lemma \ref{l_strong}-(i). Substituting $\beta = 1+\frac{\mu}{\|\cQ\|}$ into  \cite[Proposition 3.4-(v)]{fxue_1}, we obtain that 
 $\big\| \bb^k - \bb^{k+1} \big\|_\cQ$  is globally $r$--linearly convergent w.r.t. $\big\| \bb^0 - \bbstar \big\|_\cQ$:
\[
    \big\| \bb^k - \bb^{k+1} \big\|_\cQ
   \le  \sqrt{\frac{2\mu}{\|\cQ\|}}  \cdot  \Big(1+ 
   \frac{2\mu}{\|\cQ\|} \Big)^{-\frac{k-1}{4} }
\big\| \bb^0 - \bbstar \big\|_\cQ
\]
if $\mu \ge \frac{ \sqrt{5} -1} {4} \|\cQ\|$.  
\end{remark}

We further develop the convergence results in terms of objective value.
\begin{proposition}[Convergence in terms of objective value] \label{p_ppa_strong_obj}
Under the conditions of Theorem  \ref{t_ppa}, if $h$ is $\mu$--strongly convex,  then, the following hold.

{\rm (i) [Basic convergence]} $ h(\bb^{k }) - h(\bbstar) \le \frac{\mu} { 2  \|\cQ\|}   \cdot 
\frac{1}{  ( 1+ \frac{\mu} {\|\cQ\| }  )^k - 1 } \cdot 
  \big\|\bb^0-\bbstar \big\|_\cQ^2 $.

{\rm (ii) [$r$-linear convergence]} If $\mu \ge 
\frac{\sqrt{5} +1} {2} \|\cQ\|$,  $ h(\bb^{k })$ is globally   $r$-linearly convergent to 
$ h(\bbstar)$:
\[
 h(\bb^{k }) - h(\bbstar) \le \frac{\mu} { 2
\|\cQ\| }   \cdot \Big( 1+ \frac {\mu} {\|\cQ\|} \Big)^
{-\frac{k}{2}} 
 \big\|\bb^0-\bbstar \big\|_\cQ^2 
\]
The above $r$--linear convergence is also locally satisfied, for $k \ge \frac {\ln ((1+\sqrt{5}) /2) } 
{ \ln \sqrt{ 1+\frac{\mu}{\|\cQ\|} }  }$, if $\mu \in \ \big] 0,  
\frac{\sqrt{5} +1} {2} \|\cQ\| \big[ $. 
\end{proposition}

\begin{proof}
(i) First, noting that $h$ is $\mu$-strongly convex, by the $\mu$-strong convexity of $\cA$, we have, by \cite[Theorem 5.24-(ii)]{beck_book}, that:
\[
h(\bb_1) \ge h(\bb_2) + \langle \partial h(\bb_2), \bb_1 - \bb_2 \rangle + \frac{\mu}{2} \|\bb_1 - \bb_2 \|^2,\quad \forall 
(\bb_1, \bb_2) \in \cH \times \cH
\]
Thus, \eqref{temp_1} and \eqref{temp_2} can be modified as:
\be \label{xxa}
\left\{ \begin{array} {lll}
h(\bb^\star) - h(\bb^{k+1}) &\ge &  
 \frac{1}{2} \big\| \bb^{k+1} - \bb^k \big\|_\cQ^2
+ \frac{1}{2} \big(1+\frac{\mu}{\|\cQ\|}\big) 
\big\| \bb^{k+1} - \bb^\star \big\|_{\cQ}^2
-  \frac{1}{2} \big\| \bb^{k } - \bb^\star \big\|_\cQ^2 \\
h(\bb^k) - h(\bb^{k+1}) &\ge &  
 \big\| \bb^{k+1} - \bb^k \big\|_{\cQ+\frac{\mu}{2}\cI}^2
\end{array} \right.
\ee
where the first inequality of \eqref{xxa}  uses  $ \|\cdot\|^2 \ge \frac{1} {\|\cQ\|}  \|\cdot \|_\cQ^2$.

Denoting $v_k =  h( \bb^{k})- h(\bb^\star)   $, multiplying the second of \eqref{xxa} by $\alpha_k \ge 0$, and adding the first of \eqref{xxa}, one obtains:
\begin{eqnarray}
\alpha_k v_k - (\alpha_k +1) v_{k+1}   \ge 
\frac{1}{2} \big( 1+ \frac{\mu}{\|\cQ\|} \big) 
\| \bb^{k+1} - \bbstar\|^2_\cQ 
- \frac{1}{2} \| \bb^{k} - \bbstar\|^2_\cQ 
\nonumber 
\end{eqnarray}
Multiplying by $\beta_k \ge 0$ obtains:
\be \label{xa1}
2 \underbrace{ \beta_k \alpha_k}_{t_k} v_k - 
2 \underbrace{  \beta_k (\alpha_k +1)}_{t_{k+1}} v_{k+1} 
\ge \underbrace{   \beta_k \big( 1+ \frac{\mu}{\|\cQ\|} \big)  }
_{  \beta_{k+1}} \| \bb^{k+1} - \bbstar\|^2_\cQ 
- \beta_k \| \bb^{k} - \bbstar\|^2_\cQ 
\ee 
which yields:
\be \label{xa2}
2  t_{k+1}  v_{k+1} + 
\beta_{k+1}  \| \bb^{k+1} - \bbstar\|^2_\cQ 
 \le  2  t_k  v_k +   \beta_k \| \bb^{k} - \bbstar\|^2_\cQ
 \le \cdots \le   2  t_1  v_1 +   \beta_1 \| \bb^{1} - \bbstar\|^2_\cQ
\ee 
By the first inequality of \eqref{xxa}, we have:
\be \label{xa3}
 2  t_1  v_1  =  2 t_1  (h(\bb^1) - h(\bbstar))
 \le t_1 \|\bb^0-\bbstar\|_\cQ^2 - \big(1+ \frac{\mu}{\|\cQ\|}\big) t_1   \|\bb^1 - \bbstar\|_\cQ^2 -  t_1 
  \|\bb^1-\bb^0\|_\cQ^2 
\ee
Substituting \eqref{xa3} into \eqref{xa2} yields:  
\begin{eqnarray} \label{xa5}
2t_k v_k & \le & 2t_1v_1 +\beta_1 \big\|\bb^1 - \bbstar 
\big\|_\cQ^2 
\nonumber \\
& \le &  t_1 \|\bb^0-\bbstar\|_\cQ^2 - \big(1+ \frac{\mu}{\|\cQ\|}\big) t_1   \|\bb^1 - \bbstar\|_\cQ^2 + \beta_1 
  \|\bb^1-\bbstar\|_\cQ^2 
\nonumber \\
& \le &  t_1 \|\bb^0-\bbstar\|_\cQ^2 - \big(t_1+ \frac{\mu}{\|\cQ\|} t_1 - \beta_1 \big)    \|\bb^1 - \bbstar\|_\cQ^2 
\end{eqnarray}

Now, we evaluate $t_{k}$.  From \eqref{xa1}, we have:
\[
 \beta_{k+1} =  \big( 1+ \frac{\mu}{\|\cQ\|} \big) \beta_k  ; \quad
 \alpha_{k+1} = \frac{1}{  1+ \frac{\mu}{\|\cQ\|}   } (\alpha_k+1)
\]
which leads to:
\[
\alpha_k = \Big( \alpha_0 - \frac{\eta}{1-\eta} \Big) \eta^k + 
\frac{\eta}{1-\eta};\quad
\beta_k = \beta_0 \eta^{-k}
\]
where   $\eta := (1+ \frac{\mu}{\|\cQ\|})^{-1}$. 

Back to \eqref{xa5}. It is easy to check that $ t_1+ \frac{\mu}{\|\cQ\|} t_1 \ge \beta_1$, as long as $\alpha_0 \ge 0$. Thus, \eqref{xa5} becomes:
\[
v_k \le  \frac{ t_1}{2t_k} \big\|\bb^0-\bbstar \big\|_\cQ^2 
= \frac{ \alpha_1 \beta_1} {2 \alpha_k \beta_k} \big\|\bb^0-\bbstar \big\|_\cQ^2 = \underbrace{ 
 \frac{ \alpha_0+1 } {2 \big( \alpha_0 + \frac{\|\cQ\|}{\mu} (\eta^{-k}-1) \big) } }_\nu  \big\|\bb^0-\bbstar \big\|_\cQ^2 
\]
Since $\nu$ is increasing with $\alpha_0$,  the best estimate of $v_k$ follows by letting  $\alpha_0 = 0$.

\vskip.3cm
(ii)  It is easy to check that $ (1+\frac{\mu}{\|\cQ\|})^k-1 \ge 
 (1+\frac{\mu}{\|\cQ\|})^{k/2}$, if $k \ge  \frac {\ln ((1+\sqrt{5})/2 ) } { \ln \sqrt{ 1+\frac{\mu}{\|\cQ\|} } } $. Then, the $r$-linear convergence immediately follows from (i).
\end{proof}

\section{A relaxed metric resolvent}
\label{sec_relax}
We further consider  a relaxed version of operator $\cT$, defined as:
\be \label{t_relax}
\cT_\gamma := \cI +\gamma (\cT-\cI)
\ee
where $\cT$ is given by  \eqref{t} with $\cQ \in \cM_\cS^+$, $\gamma$ is a relaxation parameter. $\cT_\gamma$ can also be expressed in terms of $\cR = \cI-\cT$ as  $\cT_\gamma  = \cI - \gamma \cR$.  In particular,  $\cT_\gamma = \cT$, if $\gamma = 1$ (no relaxation).

\subsection{Nonexpansive properties}
The nonexpansiveness of $\cT_\gamma$ is presented in Lemma \ref{l_t_gamma}.
\begin{lemma}  \label{l_t_gamma}
If $\gamma \in \ ]0, 2[$,  the following hold. 

{\rm (i)} $\gamma \cR \in \cF^\cQ_{\frac{\gamma}{2-\gamma}, \frac{2-\gamma}{2}}$, 
$ \cT_\gamma \in \cF^\cQ_{1, \frac{\gamma}{2}}$.  

{\rm (ii)} $\gamma \cR$ is $\cQ$--based $\frac{1}{\gamma}$--cocoercive. 
 
{\rm (iii)} $\cT_\gamma$ is $\cQ$--based 1--cocoercive, if $\gamma \in \ ] 0, 1]$. 

{\rm (iv)} If $\gamma \in \ ] 0, 1]$,  $\cT_\gamma$ is $\cQ$--firmly nonexpansive; if $\gamma \in \ ]1, 2[$,  $\cT_\gamma$ is $\cQ$--nonexpansive, but not $\cQ$--firmly nonexpansive.
\end{lemma}

\begin{proof}
(i) Lemma \ref{l_t}--(ii) and \cite[Theorem 2.8--(iv)]{fxue_1}. 

(ii)--(iii) \cite[Theorem 2.8--(iii)]{fxue_1}.

(iv) \cite[Theorem 2.8--(ii)]{fxue_1}.
\end{proof}

\vskip.2cm
If $\cA$ is $\mu$--strongly monotone, $\cT_\gamma$ has the following properties.
\begin{lemma}  \label{l_t_gamma_strong}
If $\gamma \in \ \big] 0, 1+\frac{\|\cQ\|} {2\mu +\|\cQ\|} \big[$,  the following hold. 

{\rm (i)} $\gamma \cR \in \cF^\cQ_{\frac{ \gamma \|\cQ\|}
{ 2(1-\gamma) \mu + (2-\gamma) \|\cQ\|}, 
\frac { 2(1-\gamma) \mu + (2-\gamma) \|\cQ\|} 
{2\mu+ 2 \|\cQ\|}  }$, 
$ \cT_\gamma \in  \cF^\cQ_{\frac{  \|\cQ\|}
{ 2  \mu + \|\cQ\|},    
\frac { \gamma( 2  \mu +  \|\cQ\|)} {2\mu+ 2 \|\cQ\|}  }$.  

{\rm (ii)} $\gamma \cR$ is $\cQ$--based $\frac{\|\cQ\| +(1-\gamma^2) \mu } {\gamma \|\cQ\| + 2\gamma
(1-\gamma) \mu}$--cocoercive. In particular, if 
 $\gamma  \in \ ]0, 1]$, $\gamma \cR$ is $\cQ$--firmly nonexpansive.
   
{\rm (iii)} If $\gamma  \in \ ] 0, 1]$,  $\cT_\gamma$ is $\cQ$--based 
$\frac{1}{2} (1+ \frac{2\mu+\|\cQ\|} 
{\|\cQ\| +2 (1-\gamma) \mu})$--cocoercive, and  $\cQ$--firmly nonexpansive. 
\end{lemma}

\begin{proof}
(i) First,  $\cT \in \cF^\cQ_{\frac{\|\cQ\|}
{\|\cQ\| +2\mu},  \frac{2\mu +\|\cQ\|} {2\mu+2\|\cQ\|} }$ and $\cR \in \cF^\cQ_{1,  \frac{\|\cQ\|} {2\mu+2\|\cQ\|} }$ by Lemma \ref{l_strong}--(ii). Then,  (i) follows from  Lemma \ref{l_t}--(ii) and \cite[Theorem 2.8--(iv)]{fxue_1}. 

\vskip.2cm
(ii)--(iii) \cite[Theorem 2.8--(ii), (iii)]{fxue_1}. 
\end{proof}

\subsection{Krasnosel'skii-Mann iteration}
The scheme is given by:
\be \label{km_it}
\bb^{k+1} := \cT_\gamma \bb^k
\ee
where $\cT_\gamma$ is defined in \eqref{t_relax}. This is the Krasnosel'skii-Mann iteration of $\cT$, also the Banach-Picard iteration of $\cT_\gamma$.  The convergence is given  below.

\begin{theorem}[Convergence in terms of $\cQ$--based distance] \label{t_km}
Let $\bb^0\in \cH$, $\{\bb^k\}_{k \in \N}$ be a sequence generated by \eqref{km_it}. If $\gamma \in \  ]0, 2[$, then, the following hold.

{\rm (i)} $\cT_\gamma$ is $\cQ$--asymptotically regular.

{\rm (ii) [Basic convergence]}  There exists $\bbstar \in \Fix \cT$, such that $\bb^k \rightarrow \bbstar$, as $k\rightarrow \infty$.

{\rm (iii) [Sequential error]}   $\|\bb^{k+1 } -\bb^{k} \|_\cQ$ has the pointwise sublinear convergence rate of $\cO(1/\sqrt{k})$:
\[
\big\|\bb^{k +1} -\bb^{k} \big\|_\cQ
\le \frac{1}{\sqrt{k+1}}   \sqrt{ \frac{\gamma}{2-\gamma} }
\big\|\bb^{0} -\bb^\star \big\|_\cQ,
\forall k \in \N
\]
\end{theorem}
\begin{proof}
By Lemma \ref{l_t_gamma}-(i), $ \cT_\gamma \in \cF^\cQ_{1, \frac{\gamma}{2}}$. All of the results follows from \cite[Theorem 3.3]{fxue_1}, by substituting $\xi=1$ and $\alpha=\gamma/2$; or directly from \cite[Corollary 3.5]{fxue_1}.
\end{proof}

Under the condition of $\mu$--strongly monotone $\cA$,
the convergence properties of  \eqref{km_it} are presented as follows. 
\begin{proposition}[Convergence in terms of $\cQ$--based distance] \label{p_km_mu}
Let $\bb^0\in \cH$, $\{\bb^k\}_{k \in \N}$ be a sequence generated by \eqref{km_it}. If $\cA$ is  $\mu$--strongly monotone, $\gamma \in \ \big] 0, 1+ \frac{\|\cQ\|}
{2\mu +\|\cQ\|} \big[$, then, the following hold.

{\rm (i) [Sequential error]}   $\|\bb^{k+1 } -\bb^{k} \|_\cQ$ has the pointwise sublinear convergence rate of $\cO(1/\sqrt{k})$:
\[
\big\|\bb^{k +1} -\bb^{k} \big\|_\cQ
\le \frac{1}{\sqrt{k+1}} \cdot 
\sqrt{ \frac{\gamma(2\mu+\|\cQ\|)} 
 {2\mu (1-\gamma)  + (2-\gamma) \|\cQ\|} }
 \big\|\bb^{0} -\bb^\star \big\|_\cQ,
\forall k \in \N
\]

{\rm (ii) [$q$--linear convergence]} Both $\|\bb^{k} -\bbstar \|_\cQ$ and $\|\bb^{k} -\bb^{k+1} \|_\cQ$  are $q$--linearly convergent with the rate of $\sqrt{ 1 - \frac{2 \gamma \mu }  {2\mu +\|\cQ\|} } $.

{\rm (iii) [$r$--linear convergence]}  If $\|\cQ\| < (\sqrt{5}+1) \mu$,  $\gamma \in \ \big] \frac{3-\sqrt{5} }{4} (2 + \frac{\|\cQ\|} {\mu} ), 1 \big[ $, 
 $\big\| \bb^k - \bb^{k+1} \big\|_\cQ$  is globally $r$--linearly convergent w.r.t. $\big\| \bb^0 - \bbstar \big\|_\cQ$:
\[
    \big\| \bb^k - \bb^{k+1} \big\|_\cQ
   \le \frac{\gamma}{1-\gamma} \cdot 
   \sqrt{  1 + \frac{ \|\cQ\|  }   {2 \mu }  }
    \cdot \Big(1- \frac{2 \gamma  \mu }
{2\mu +\|\cQ\|} \Big)^{\frac{k+1}{4} }  
\big\| \bb^0 - \bbstar \big\|_\cQ
\]
The above inequality is also locally satisfied, for  $k \ge  \frac{\ln((1+\sqrt{5})/2)}  {\ln \sqrt{ \frac{2\mu +\|\cQ\|}
{ 2(1-\gamma)\mu +\|\cQ\|}  } } - 1$, 
 if $0< \gamma <  \min \big\{ \frac{ 3 - \sqrt{5} }{4}  ( 2 +\frac{ \|\cQ\|} {\mu } ) , 1 \big\} $.  
\end{proposition}
\begin{proof}
(i) Substituting $\alpha=\frac{\gamma(2\mu+\|\cQ\|)}
{2\mu+\|\cQ\|}$ into  \cite[Theorem 3.3--(iii)]{fxue_1}.

\vskip.2cm
(ii) By Lemma \ref{l_t_gamma_strong}-(i), we have
$ \cT_\gamma \in  \cF^\cQ_{\frac{  \|\cQ\|}
{ 2  \mu + \|\cQ\|},   
\frac { \gamma( 2  \mu +  \|\cQ\|)} {2\mu+ 2 \|\cQ\|}  }$.  
Denote $\nu:= 1-\alpha + \alpha \xi^2$, which, for $\cT_\gamma$, is computed as:
 \[
\nu =  1-  \frac { \gamma (2  \mu +  \|\cQ\|)} 
{2\mu+ 2 \|\cQ\|} + \frac { \gamma (2  \mu +  \|\cQ\|)} 
{2\mu+ 2 \|\cQ\|}  \cdot \frac{  \|\cQ\|^2}
{ (2  \mu + \|\cQ\|)^2} = 1-
\frac{2 \gamma  \mu } {2\mu +\|\cQ\|}
 \]   
Then, (ii) follows by   \cite[Theorem 3.3--(iv)]{fxue_1}. 

\vskip.2cm
(iii) By  \cite[Theorem 3.3--(v)]{fxue_1}, we obtain that if $\gamma >\frac{3-\sqrt{5} }{4} (2 + \frac{\|\cQ\|} {\mu} ) $, then:
\[
    \big\| \bb^k - \bb^{k+1} \big\|_\cQ
   \le\sqrt{  \frac{2 \gamma^2 \mu  (2\mu+\|\cQ\|)  } 
   { [(2-2\gamma) \mu +(2-\gamma) \|\cQ\|] \cdot 
  [ (2-2\gamma) \mu +\|\cQ\|] }  }
    \cdot \Big( \frac{2(1-\gamma) \mu +\|\cQ\|}
{2\mu +\|\cQ\|} \Big)^{\frac{k+1}{4} }  
\big\| \bb^0 - \bbstar \big\|_\cQ
\]
which can be reduced to (iii), if $\gamma \in \ ]0,1[$. In addition, to guarantee $\frac{3-\sqrt{5} }{4} (2 + \frac{\|\cQ\|} {\mu} )  <1$, we need to require that  $\|\cQ\| < (\sqrt{5}+1) \mu$.

\vskip.1cm
Finally, note that the same results can also be reached by \cite[Corollary 3.5--(iv), (v)]{fxue_1}.
\end{proof}

\begin{remark}
Proposition \ref{p_km_mu} can be similarly deduced from \cite[Proposition 3.4]{fxue_1}. First, if $\gamma  \in \  ]0, 1]$,  $\cT_\gamma$ is $\cQ$--based 
$\frac{1}{2} (1+ \frac{2\mu+\|\cQ\|} 
{\|\cQ\| +2 (1-\gamma) \mu})$--cocoercive, by Lemma \ref{l_t_gamma_strong}--(iii). Substituting $\beta = \frac{1}{2} (1+ \frac{2\mu+\|\cQ\|} 
{\|\cQ\| +2 (1-\gamma) \mu}) $ into \cite[Proposition 3.4--(iii)]{fxue_1}, we obtain:
\[
\big\|\bb^{k +1} -\bb^{k} \big\|_\cQ
\le \frac{1}{\sqrt{k+1}} \cdot 
\sqrt{ \frac{2\mu+\|\cQ\|}  {2\mu (1-\gamma)  + \|\cQ\|} }
 \big\|\bb^{0} -\bb^\star \big\|_\cQ,
\forall k \in \N
\]
which is valid for only $\gamma \in \ ]0, 1[$. This is a slightly different result from Proposition \ref{p_km_mu}--(i).

Proposition \ref{p_km_mu}--(ii) and (iii) can be exactly obtained by \cite[Proposition 3.4--(iv), (v)]{fxue_1}.
\end{remark}

\section{Further extension: A generalized metric resolvent}
\label{sec_general}
\subsection{Scheme}
We further consider  a generalized version of operator $\cT$, defined as:
\be \label{tt_relax}
\cTtilde := \cI +\cM (\cT-\cI)
\ee
where $\cT$ is given by  \eqref{t}, $\cM$ is a correction/relaxation matrix. $\cTtilde$ can also be expressed in terms of $\cR = \cI-\cT$ as  $\cTtilde = \cI - \cM \cR$.  In particular, if $\cM = \gamma \cI$,   $\cTtilde$ reduces to the relaxed version discussed in Section \ref{sec_relax}. The fixed-point iteration of $\cTtilde$  is given by:
\be \label{gkm_it}
\bb^{k+1} := \cTtilde \bb^k
\ee
This is the generalized Krasnosel'skii-Mann iteration of $\cT$, also the Banach-Picard iteration of $\cTtilde$.  

Before the convergence analysis, we need the well-known  Opial's lemma  \cite{opial} as a basic tool for  convergence analysis. Also see \cite[Lemma 2.39]{plc_book} for the proof.
\begin{lemma} \label{l_opial} {\rm  [Opial's lemma \cite{opial}] }
Let $\{\bb^k\}_{k \in \N}$ be a sequence in $\cH$, and let $C$ be  a nonempty set $C \subset \cH$. Suppose that:

{\rm (a)} for every $\bbstar  \in C$, $\{ \|\bb^k -\bbstar\|\}_{k\in\N}$ converges, i.e. $\lim_{k\rightarrow \infty} \|\bb^k
-\bbstar \| $ exists; 

{\rm (b)} every sequential cluster point of  $\{\bb^{k}\}_{k \in \N}$ belongs to $C$. 

\noindent
Then,  $\{\bb^k\}_{k\in\N}$ converges to a point in $C$.
\end{lemma}

\subsection{Convergence analysis}
To analyze the convergence of \eqref{gkm_it}, we first rewrite \eqref{gkm_it} as a two--step inclusion form:
\be \label{gppa}
\left\lfloor \begin{array}{llll}
{\bf 0}  & \in & \cA \tilde{\bb}^{k} + \cQ (\tilde{\bb}^{k} -  \bb^k)  & \text{proximal step} \\ 
\bb^{k+1} & := & \bb^k + \cM ( \tilde{\bb}^{k} - \bb^k ) 
& \text{correction step}
\end{array}  \right.
\ee

Lemma \ref{l_gppa} presents several key ingredients, which are the `recipe' for proving the convergence results.

\begin{lemma} \label{l_gppa}
Let $\bb^\star \in \zer \cA$ and $\{\bb^k\}_{k\in\N}$ be a sequence generated by \eqref{gkm_it} or \eqref{gppa}. Denote $\cS: = \cQ \cM^{-1}$, $\cG := \cQ + \cQ^{\top} - \cM^\top \cQ$, and the operator $\cR := \cI - \cT$. If  $\cA$ is maximally monotone and  $\cS \in \cM_\cS$. Then, the following hold.

{\rm (i)}  $  \big\| \bb^{k+1} - \bbstar \big\|_\cS^2
\le   \big\| \bb^k - \bbstar \big\|_\cS^2
-  \big \| \bb^k -  \bb^{k+1}  \big \|
_{\cM^{-\top} \cG \cM^{-1} }^2 $

{\rm (ii)}  $ \big \langle     \cR \bb^k,  \cR \bb^k - \cR \bb^{k+1} 
  \big \rangle_{ \cM^\top \cS \cM}   \ge  \frac{1}{2}
\big\|   \cR \bb^k - \cR  \bb^{k+1}  \big\|
_{\cQ+\cQ^\top}^2$

{\rm (iii)}  $  \big \| \bb^k - \bb^{k+1} \big\|_\cS^2 - 
\big\| \bb^{k+1} - \bb^{k+2} \big\|_\cS^2   
 \ge   \big\|   \cR \bb^k -  \cR \bb^{k+1}  \big\|_\cG^2  $
\end{lemma}

\begin{proof}
(i) From  \eqref{gppa}, we have:
\begin{eqnarray} 
0 & \le &  \big \langle \cA  \bbtilde^{k} - \cA \bbstar,
\bbtilde^k - \bbstar \big \rangle  \quad
\text{by monotonicity of $\cA$}
\nonumber \\
&=&  \big \langle \cQ (\bb^k -  \bbtilde^{k}),
\bbtilde^k - \bbstar \big \rangle
\quad  \text{ by \eqref{gppa} and ${\bf 0} \in \cA \bbstar $ }
\nonumber \\
&=&  \big \langle \cQ \cM^{-1} (\bb^k -  \bb^{k+1}),
\bb^k + \cM^{-1} (\bb^{k+1} - \bb^k) - \bbstar \big \rangle
\quad  \text{ by \eqref{gppa} }
\nonumber \\
&=&  \big \langle \bb^k -  \bb^{k+1}, \bb^k - \bbstar
\big \rangle_\cS
- \frac{1}{2} \big \| \bb^k -  \bb^{k+1}  \big \|_{\cM^{-\top} \cS + \cS \cM^{-1} }^2 
\quad  \text{ by $\cS = \cQ \cM^{-1}$}
\nonumber \\
&=& \frac{1}{2} \big\| \bb^k - \bbstar \big\|_\cS^2
- \frac{1}{2} \big\| \bb^{k+1} - \bbstar \big\|_\cS^2
- \frac{1}{2} \big \| \bb^k -  \bb^{k+1}  \big \|
_{\cM^{-\top} \cG \cM^{-1} }^2 \quad \text{by definition of $\cG$}
\nonumber 
\end{eqnarray}

\vskip.2cm
(ii) By \cite[Lemma 2.6-(v)]{fxue_1}, we have:
\[
 \big \langle     \bb^k - \bb^{k+1} , 
 \cR \bb^k - \cR \bb^{k+1}  \big \rangle_{\cQ^\top}  \ge 
\big\|   \cR \bb^k - \cR  \bb^{k+1}  \big\|_{\cQ^\top}^2
= \frac{1}{2} \big\|   \cR \bb^k - \cR  \bb^{k+1}  \big\|_{\cQ + \cQ^\top}^2
\]
Then, (ii) follows from  $\bb^k - \bb^{k+1} = \cM\cR \bb^k$ by \eqref{gppa} and $\cQ^\top = \cM^\top \cS$.

\vskip.2cm
(iii) From Lemma \ref{l_gppa}-(ii), we have:
\begin{eqnarray}  \label{dd}
& &  \big \| \bb^k - \bb^{k+1} \big\|_\cS^2 - 
\big\| \bb^{k+1} - \bb^{k+2} \big\|_\cS^2   
\nonumber \\
& = &  \big \|\cM \cR \bb^k \big\|_\cS^2 - 
\big\| \cM \cR\bb^{k+1} \big\|_\cS^2  
\quad \text{by \eqref{gppa} } 
\nonumber \\
&=& 2 \big \langle  \cR \bb^k , 
 \cR \bb^k - \cR \bb^{k+1}   \big \rangle_{ \cM^\top \cS \cM}   
- \big\| \cR\bb^k - \cR\bb^{k+1} \big\|^2
_{\cM^\top   \cS \cM }  
\nonumber \\ 
& \ge &  \big\|   \cR \bb^k -  \cR \bb^{k+1}  \big\|_{\cQ +\cQ^\top  - \cM^\top \cS\cM }^2  \qquad 
\text{by Lemma \ref{l_gppa}-(ii) }   
\nonumber 
\end{eqnarray}
This completes the proof.
\end{proof}

\vskip.3cm
Then, the convergence properties of \eqref{gkm_it} are given by the next theorem. 

\begin{theorem}[Convergence in terms of solution distance] \label{t_gppa}
Let $\{\bb^k\}_{k\in\N}$ be a sequence generated by \eqref{gkm_it}. If $\cS$ and $\cG$ defined in Lemma \ref{l_gppa} satisfy $\cS, \cG \in \cM_\cS^{++}$, then the following hold.

{\rm (i) [Basic convergence]} There exists $\bbstar\in \zer \cA$, such that $\bb^k \rightarrow \bbstar$, as $k \rightarrow \infty$. 

{\rm (ii) [Sequential convergence]}  $\|  \bb^{k } - \bb^{k+1 } \|_\cS$ has the  convergence rate of $\cO(1/\sqrt{k})$, i.e.
\[
\big\| \bb^{k+1 } - \bb^{k }  \big\|_\cS 
\le \frac{1}{ \sqrt{k+1 } } \sqrt{\frac
{\lambda_{\max}(\cS )} 
{\lambda_{\min}(\cM^{-\top} \cG \cM^{-1})}  }  
\big\|\bb^{0} -\bb^\star \big\|_\cS,
\quad k \in \N
\]
\end{theorem}

\begin{proof} 
(i)  From Lemma \ref{l_gppa}-(i), the conditions $\cS,\cG \in \cM_\cS^{++}$ guarantee that 
 $\{ \| \bb^{k} - \bb^\star \|_\cS \}_{k\in\N} $ is non-increasing, and bounded from below (always being non-negative), and thus, convergent, i.e. $\lim_{k\rightarrow \infty} \| \bb^{k} - \bb^\star \|_\cS  $ exists. Thus, the condition (a) of Opial's lemma (see Lemma \ref{l_opial}) is satisfied.

On the other hand, summing up the inequality of Lemma \ref{l_gppa}-(i) from $k=0$ to $K-1$ yields:
\begin{eqnarray}  
\big \|\bb^K - \bb^\star \big\|_\cS^2  
& \le &  
\big\| \bb^{0} - \bb^\star \big\|_\cS^2 
-\sum_{k=0}^{K-1} \big\| \bb^k - \bb^{k+1} \big\|
_{\cM^{-\top} \cG \cM^{-1}}^2  \nonumber  \\
&=&   \big\| \bb^{0} - \bb^\star \big\|_\cS^2 
-\sum_{k=0}^{K-1} \big\| \bb^k - \bbtilde^{k} \big\|
_{  \cG  }^2  \quad \text{by \eqref{gppa}}
\nonumber 
\end{eqnarray}
Taking $K \rightarrow \infty$, we have: $\sum_{k=0}^{\infty} \big\| \bb^k - \bbtilde^{k} \big\|_{\cG}^2 \le \big\| \bb^{0} - \bb^\star \big\|_\cS^2 < \infty$. The condition of $\cG \succ \bf 0$  implies that  $\bb^k-\bbtilde^k \rightarrow \bf 0$, as $k \rightarrow \infty$. Since $\cQ(\bb^k - \bbtilde^k) \in \cA \bbtilde^k$ from \eqref{gppa}, we have: $\dist(\cA\bbtilde^k, {\bf 0}) \rightarrow  0$. Thus, the cluster point of $\{\bb^k\}_{k\in\N}$ belongs to $\zer \cA$. The condition (b) of Opial's lemma is satisfied. Finally, (i) is reached by Lemma \ref{l_opial}.

\vskip.2cm 
(ii) In view  of Lemma \ref{l_gppa}-(i), we have:
\[
\big\| \bb^{k+1} -\bb^\star \big\|_\cS^2 \le 
\big\| \bb^{k} - \bb^\star \big\|_\cS^2 - \frac{\lambda_{\min}( \cM^{-\top} \cG \cM^{-1})}
{\lambda_{\max}(\cS )} \big\| \bb^{k } - \bb^{k+1} \big\|_\cS^2  
\] 
 where $\lambda_{\max}$ and $\lambda_{\max}$ denote the largest and smallest eigenvalues of a matrix. 

On the other hand,  the sequence  $\{\| \bb^{k} - \bb^{k+1} \|_\cS\}_{k\in\N} $ is non-increasing, if   $\cS,\cG\in \cM_\cS^{++}$,  by Lemma \ref{l_gppa}-(iii). Finally, (ii) is obtained, following the similar proof of Theorem \ref{t_ppa}-(ii).
\end{proof}

\begin{remark}
Lemma \ref{l_gppa} and Theorem \ref{t_gppa} do not require $\cQ \in \cM_\cS^+$. The  variable metric used in this contraction property is $\cS = \cQ \cM^{-1}$, rather than $\cQ$.  Also notice that $\cG$ and $\cM^{-\top} \cG \cM^{-1}$ are symmetric, if  $\cS$ is symmetric. 
\end{remark}

\section{Applications to first-order operator splitting  algorithms}
\label{sec_eg}
The focus of this part is to show that many popular operator splitting algorithms can be expressed by the (generalized) metric resolvent, and thus enjoy the corresponding convergence properties given in Sections \ref{sec_resolvent} and \ref{sec_relax}. The proposed metric resolvent provides a unified treatment of the operator splitting methods.

\subsection{The  ADMM algorithms}
\label{sec_admm}
ADMM is one of the most commonly used algorithms for solving the  structured constrained optimization \cite{boyd_admm}:
\[
\min_{\bx,\bu} f(\bx) +g(\bu),\quad
\text{s.t.}\ \ \bA\bx+\bB\bu = \bc
\] 
where  $\bx \in \R^N$, $\bu \in \R^L$, $\bA: \R^N \mapsto \R^M$,  $\bB: \R^L \mapsto \R^M$,  $f \in \Gamma(\cH): \R^N \mapsto \ ]-\infty, +\infty]$, $g \in \Gamma(\cH): \R^L \mapsto \ ]-\infty, +\infty]$.
Two typical ADMM  algorithms are listed here to show the corresponding fixed-point iterations of the metric resolvent.

\begin{example} [Relaxed-ADMM] \label{e_admm_2}
The relaxed-ADMM is given as  \cite[Eq.(3)]{fang_2015}:
\be \label{admm}
\left\lfloor \begin{array}{lll}
\bx^{k+1} & := &  \arg \min_\bx   f(\bx) +
\frac{\tau}{2} \big\| \bA \bx +\bB \bu^k - \bc - \frac{1}{\tau}   \bs^k    \big\|^2  \\
\bu^{k+1} & := & \arg \min_\bu   g(\bu) + \frac{\tau}{2} \big\|\bB (\bu - \bu^k) +  \gamma (\bA \bx^{k+1}
+  \bB \bu^{k} - \bc ) - \frac{1}{\tau}  \bs^k \big\|^2 \\
\bs^{k+1} & := &   \bs^k -\tau \bB (\bu^{k+1} - \bu^k) 
  - \tau\gamma (\bA \bx^{k+1} + \bB \bu^{k} - \bc) 
\end{array} \right. 
\ee
The standard ADMM/DRS can be recovered by letting $\gamma=1$  \cite{boyd_admm}.

It can be shown that \eqref{admm} is essentially the Banach-Picard iteration of $\cTtilde$ \eqref{tt_relax}, with:
\[
\bb^k = \begin{bmatrix}
\bx^k \\ \bu^k \\ \bs^k \end{bmatrix}; \quad
\cA: \bb \mapsto  \begin{bmatrix}
\partial f & \bf 0 & -\bA^\top  \\
\bf 0 & \partial g   & -\bB^\top \\
\bA & \bB  & \bf 0 
\end{bmatrix} \bb - \begin{bmatrix}
\bf 0 \\ \bf 0 \\ \bc
\end{bmatrix}
\]
\[
\cQ =   \begin{bmatrix}
\bf 0 & \bf 0 & \bf 0 \\
\bf 0 & \tau \bB^\top \bB   & (1-\gamma)\bB^\top \\
\bf 0 & - \bB  & \frac{1}{\tau} \I_M
\end{bmatrix};\quad
\cM =  \begin{bmatrix}
\I_N & \bf 0 & \bf 0 \\
\bf 0 & \I_L & \bf 0 \\
\bf 0 &  -\tau\bB  & \gamma \I_M \\
\end{bmatrix}
\]
\end{example}

\begin{example} [Proximal-ADMM] \label{e_admm_3}
The proximal-ADMM is given as  \cite[modified SPADMM]{lxd_2016}, \cite[Eq.(10)]{admm_tmi}:
\[
\left\lfloor \begin{array}{lll}
\bx^{k+1} & := &  \arg \min_\bx   f(\bx) +
\frac{\tau}{2} \big\| \bA \bx +\bB \bu^k - \bc - \frac{1}{\tau}   \bs^k    \big\|^2  + \frac{1}{2} \big\|  \bx  -   \bx^k    \big\|_{\bP_1}^2  \\
\bu^{k+1} & := & \arg \min_\bu   g(\bu) + \frac{\tau}{2} \big\|\bA \bx^{k+1} + \bB \bu - \bc - \frac{1}{\tau}  \bs^k \big\|^2 
 + \frac{1}{2} \big\|  \bu - \bu^k  \big\|_{\bP_2}^2  \\
\bs^{k+1} & := &   \bs^k -\tau (\bA \bx^{k+1} + 
 \bB \bu^{k+1}  - \bc) 
\end{array} \right. 
\]
which is also the fixed-point iteration of $\cTtilde$ \eqref{tt_relax} with $\bb^k$ and $\cA$ defined in Example \ref{e_admm_2}, and
\[
\cQ =   \begin{bmatrix}
\bP_1 & \bf 0 & \bf 0 \\
\bf 0 & \bP_2 + \tau \bB^\top \bB   & \bf 0  \\
\bf 0 & - \bB  & \frac{1}{\tau} \I_M
\end{bmatrix};\quad
\cM =  \begin{bmatrix}
\I_N & \bf 0 & \bf 0 \\
\bf 0 & \I_L & \bf 0 \\
\bf 0 &  -\tau\bB  & \I_M \\
\end{bmatrix}
\]
\end{example}

\subsection{The PDHG algorithms}
\label{sec_pdhg}
Consider the problem:
\be \label{problem2}
\min_\bx\  f(\bx) + g(\bA \bx)
\ee 
where  $\bx \in \R^N$,  $\bA: \R^N \mapsto \R^L$,   $f \in \Gamma(\cH): \R^N \mapsto \ ] -\infty, +\infty]$, $g \in \Gamma(\cH): \R^L \mapsto\  ]-\infty, +\infty]$.   The PDHG generally first reformulates \eqref{problem2} as (obtained by Legendre-Fenchel transform \cite[Chapter 11]{rtr_book_2}):
\be \label{problem_pd}
\min_\bx \max_\bs \cL(\bx,\bs) := 
 f(\bx) + \langle \bs, \bA \bx \rangle  - g^*(\bs)
\ee
and then performs alternating update (by gradient descent) between primal $\bx$ and dual variable $\bs$. Two simple examples are listed here with their associated metric resolvent, which have been shown in \cite{hbs_siam_2012_2, hbs_jmiv_2017}  based on variational inequality.

\begin{example} \label{eg_pdhg} \cite[PDHGMp]{esser}
The basic PDHG, for solving \eqref{problem_pd}, is:
\be \label{pdhg}
\left\lfloor \begin{array}{llll}
\bs^{k+1}   & := &  \prox_{\sigma g^*}  \big( \bs^k +\sigma 
\bA \bx^k   \big) & \text{\rm dual step} \\
\bx^{k+1}   & := &  \prox_{\tau f}  \big( \bx^k - \tau   
\bA^\top (2\bs^{k+1} -  \bs^k )   \big) & 
\text{\rm primal step}   
\end{array} \right. 
\ee
which an be written as a simple fixed--point iteration of $\cT$ \eqref{t}: 
\[
\begin{bmatrix}  \bs^{k+1} \\  \bx^{k+1}
\end{bmatrix} = \bigg( \underbrace{  \begin{bmatrix}
\partial g^* & -\bA \\
\bA^\top  & \partial f    \end{bmatrix} }_\cA
 + \underbrace{    \begin{bmatrix}
\frac{1}{ \sigma} \I_L & \bA \\
\bA^\top   & \frac{1}{ \tau} \I_N 
\end{bmatrix} }_\cQ  \bigg)^{-1} 
\underbrace{  \begin{bmatrix}
\frac{1}{ \sigma} \I_L & \bA \\
\bA^\top   & \frac{1}{ \tau} \I_N 
\end{bmatrix} }_\cQ   \begin{bmatrix}  
 \bs^k \\   \bx^{k}
\end{bmatrix}
\]
\end{example}

\begin{example} \label{eg_pdhg_2} \cite[PDHGMu]{esser} 
Another form of PDHG is:
\be \label{pdhg_other}
\left\lfloor \begin{array}{llll}
\bs^{k+1}   & := &  \prox_{\sigma g^*}  \big( \bs^k +\sigma 
\bA (2 \bx^k - \bx^{k-1} )  \big) & \text{\rm dual step} \\
\bx^{k+1}   & := &  \prox_{\tau f}  \big( \bx^k - \tau   
\bA^\top \bs^{k+1}   \big) & 
\text{\rm primal step}   
\end{array} \right. 
\ee
which is the fixed--point iteration of $\cT$ \eqref{t}\footnote{Note that we use a mismatch of iteration indices between $\bx$ and $\bs$: $\bb^k := (\bs^k, \bx^{k-1})$. This technique can also be found in \cite{bot_2015}.}:
\[
\begin{bmatrix}  \bs^{k+1} \\  \bx^k
\end{bmatrix} = \bigg( \underbrace{  \begin{bmatrix}
\partial g^* & -\bA \\
\bA^\top  & \partial f    \end{bmatrix} }_\cA
 +   \underbrace{ \begin{bmatrix}
\frac{1}{ \sigma} \I_L &  -\bA  \\
-\bA^\top   & \frac{1}{ \tau} \I_N 
\end{bmatrix} }_\cQ \bigg)^{-1}  
\underbrace{   \begin{bmatrix}   \frac{1}{ \sigma} \I_L &  -\bA  \\
-\bA^\top   & \frac{1}{ \tau} \I_N 
\end{bmatrix} }_\cQ  \begin{bmatrix}  
 \bs^k \\   \bx^{k-1}
\end{bmatrix}
\]
\end{example}

\subsection{Other examples}
\label{sec_other}
Other classes of algorithms can also be viewed as the applications of  metric resolvent.   Let us now consider a typical  optimization problem with a linear equality constraint:
\[
\min_\bx h(\bx),\qquad \text{s.t.\ } \bA\bx = \bc 
\] 
where $h\in \Gamma(\cH): \R^N \mapsto \ ]-\infty,+\infty]$, $\bA: \R^N \mapsto \R^M$. It can be solved by the following examples.

\begin{example} [Basic Augmented Lagrangian Method (ALM)]
The ALM   is (see \cite[Eq.(1.2)]{mafeng_2018} and \cite[Eq.(7.2)]{taomin_2018} for example):
\be \label{alm}
\left\lfloor \begin{array}{lll}
\bx^{k+1}   & := &    \arg \min_\bx  h(\bx) +  \frac{\tau}{2}
\big\| \bA\bx - \bc -  \frac{1}{\tau} \bs^k  \big\|^2   \\
\bs^{k+1}  & :=  &  \bs^k - \tau  ( \bA \bx^{k+1} - \bc ) 
\end{array} \right. 
\ee
which can be expressed in terms of metric resolvent:
\[ 
\bb^k = \begin{bmatrix}  \bx^k    \\ \bs^k    \end{bmatrix}; \quad
\cA: \bb \mapsto  \begin{bmatrix}
 \partial h  &  -\bA^\top  \\  \bA &  \bf 0
    \end{bmatrix}  \bb - \begin{bmatrix}
\bf 0 \\ \bc   \end{bmatrix}; \quad 
\cQ =    \begin{bmatrix}
\bf 0    & \bf 0   \\  \bf 0  & \frac{1}{\tau} \I_M  
\end{bmatrix} 
\]
\end{example}

\begin{example} [Linearized ALM]
The linearlized ALM is given as \cite{yang_yuan_2013}: 
\be \label{lalm}
\left\lfloor \begin{array}{lll}
\bx^{k+1}   & := &  \arg \min_\bx   h(\bx) + \frac{\rho } {2} 
\big\| \bx - \bx^k +  \frac{1}{ \rho}  \bA^\top \big( \tau  
(\bA \bx^k -\bc) -   \bs^k \big) \big\|^2  \\
\bs^{k+1}  & :=  &  \bs^k - \tau  
( \bA \bx^{k+1} - \bc) 
\end{array} \right. 
\ee
which can be written in a metric resolvent form:
\[ 
\bb^k = \begin{bmatrix}  \bx^k    \\ \bs^k    \end{bmatrix}; \quad
\cA: \bb \mapsto  \begin{bmatrix}
 \partial h  &  -\bA^\top  \\  \bA &  \bf 0
    \end{bmatrix}  \bb - \begin{bmatrix}
\bf 0 \\ \bc   \end{bmatrix}; \quad 
\cQ =    \begin{bmatrix}
\rho \I_N  -  \tau \bA^\top  \bA    & \bf 0   \\
\bf 0  & \frac{1}{\tau} \I_M  \end{bmatrix} 
\]
\end{example}

\begin{example} [Linearized Bregman algorithm  \cite{cjf_2}]
The scheme reads as (see also \cite[Eq.(1.11)]{zxq}):
\be \label{lb}
\left\lfloor \begin{array}{lll}
\bx^{k+1}   & := &  \arg \min_\bx \rho \tau  h(\bx) + 
\frac{1}{2} \big\| \bx - \rho \bA^\top  \bs^k  \big\|^2   \\
\bs^{k+1} & :=  &  \bs^k  - ( \bA\bx^{k+1} - \bc)
\end{array} \right. 
\ee
with the metric resolvent form:
\[ 
\bb^k = \begin{bmatrix}  \bx^k    \\ \bs^k    \end{bmatrix}; \quad
\cA: \bb \mapsto  \begin{bmatrix}
  \tau  \partial h + \frac{1}{\rho} \I_N - \bA^\top \bA
    &  -\bA^\top  \\  \bA &  \bf 0
    \end{bmatrix}  \bb + \begin{bmatrix}
 \bA^\top \bc \\ -\bc   \end{bmatrix}; \quad 
\cQ =    \begin{bmatrix}
\bf 0    & \bf 0   \\  \bf 0  &  \I_M  
\end{bmatrix} 
\]
\end{example}

Before ending this section, we note that more    existing algorithms, besides from the listed ones, can be reexpressed by the (generalized) metric resolvent.  We do not not enumerate all the supporting evidences here, and refer interested readers to \cite{mafeng_2018,fang_2015,
cch_2016,hbs_yxm_2018,deng_2017,bai_2018} for more examples.

\section{Conclusions}
In this paper, we investigated the nonexpansiveness of the metric resolvent and its generalized (typically, relaxed) version, from which immediately followed the convergence properties of the associated fixed-point iterations, in terms of both solution distance and objective value. In particular, it is shown that many classes of operator splitting methods can be expressed by the (generalized) metric resolvent.  This study provides a unified understanding and treatment of these  algorithms. Last, it seems interesting to analyze the accelerated algorithms under the paradigm of metric resolvent, which will be left to future work.

\bibliographystyle{siamplain} 
\bibliography{refs}

%\newpage
%\appendix
%
%\section{Key technical lemmas}
%
%For completeness, we firstly give proof of Three point property. 
%
%\subsection{Proof of the three point property}
%
%Note that $\phi(x)+D_h(x,z)$ is differentiable and convex in $x$. Using the definition of $z_+$ we have

\end{document}